\documentclass[reqno,12pt]{amsart}
\usepackage[margin=1.1in, footskip=1cm]{geometry}
\usepackage{amsmath, amsthm, amssymb, booktabs, url}
\usepackage{mathrsfs}
\usepackage[dvipsnames]{xcolor}
\pdfpagewidth=\paperwidth
\pdfpageheight=\paperheight

\newtheorem{theorem}{Theorem}[section]
\newtheorem{lemma}[theorem]{Lemma}

\newtheorem{hypothesis}[theorem]{Hypothesis}

\theoremstyle{definition}
\newtheorem{definition}[theorem]{Definition}

\theoremstyle{remark}

\numberwithin{equation}{section}

\newcommand{\mmod}[1]{\,\,({\rm{mod}}\,\,#1)}

\def\le{\leqslant} \def\ge{\geqslant}

\makeatletter
\@namedef{subjclassname@2020}{\textup{2020} Mathematics Subject Classification}
\makeatother

\begin{document}
\title[Equidistribution of polynomial sequences]{Equidistribution of polynomial sequences in function fields: 
resolution of a conjecture}
\author[J. Champagne]{J. Champagne}
\address{JC, ZG and Y-RL: Department of Pure Mathematics, University of Waterloo, 200 University Avenue West, 
Waterloo, ON, N2L 3G1, Canada}
\email{jchampag@uwaterloo.ca, z4ge@uwaterloo.ca, yrliu@uwaterloo.ca }
\author[Z. Ge]{Z. Ge}
\author[T. H. L\^e]{T. H. L\^e}
\address{THL: Department of Mathematics, University of Mississippi, 305 Hume Hall, University, MS 38677, USA}
\email{leth@olemiss.edu}
\author[Y.-R. Liu]{Y.-R. Liu}
\author[T. D. Wooley]{T. D. Wooley}
\address{TDW: Department of Mathematics, Purdue University, 150 N. University Street, West 
Lafayette, IN 47907-2067, USA}
\email{twooley@purdue.edu}
\subjclass[2020]{11J71, 11T55}
\keywords{Equidistribution, function field analogues}
\thanks{The third author is supported by NSF grant DMS-2246921 and a travel award from the Simons 
Foundation, the fourth author is supported by an NSERC discovery grant, while the fifth author is supported by NSF 
grant DMS-2502625 and Simons Fellowship in Mathematics SFM-00011955. The fifth author is grateful to the 
Institute for Advanced Study, Princeton, for hosting his sabbatical, during which period this paper was completed.}

\begin{abstract} Let $\mathbb F_q$ be the finite field of $q$ elements having characteristic $p$, and denote by 
$\mathbb K_\infty=\mathbb F_q((1/t))$ the field of formal Laurent series in $1/t$. We consider the equidistribution in 
$\mathbb T=\mathbb K_\infty/\mathbb F_q[t]$ of the values of polynomials $f(u)\in \mathbb K_\infty [u]$ as $u$ 
varies over $\mathbb F_q[t]$. Let $\mathcal K$ be a finite set of positive integers, and suppose that 
$\alpha_r\in \mathbb K_\infty$ for $r\in \mathcal K\cup \{0\}$. We show that the polynomial 
$\sum_{r\in \mathcal K\cup\{0\}}\alpha_ru^r$ is equidistributed in $\mathbb T$ whenever $\alpha_k$ is irrational for 
some $k\in \mathcal K$ satisfying $p\nmid k$, and also $p^vk\not\in \mathcal K$ for any positive integer $v$. This 
conclusion resolves in full a conjecture made jointly by the third, fourth and fifth authors.
\end{abstract}

\maketitle

\section{Introduction}
Our focus in this paper lies on an analogue for function fields of the equidistribution theory initiated by Weyl 
\cite{Wey1916}. Write $\mathbb Z^+$ for the set of positive integers and $\{ a\}$ for the fractional part of a real 
number $a$, by which we mean $a-\lfloor a\rfloor$, where $\lfloor a\rfloor$ denotes the largest integer not 
exceeding $a$. We recall that in the classical setting, a sequence $(s_n)_{n=1}^\infty$ of real numbers is said to be 
{\it equidistributed modulo $1$} if, for any interval $[\alpha,\beta]\subset [0,1)$, we have
\[
\lim_{N\rightarrow \infty}N^{-1}\text{card}\{ n\in [1,N]\cap \mathbb Z^+:\{ s_n\}\in [\alpha,\beta]\}=\beta-\alpha .
\]
The celebrated observation of Weyl is that whenever $f(u)=\sum_{r=0}^k\alpha_ru^r$ is a polynomial with real 
coefficients in which one, at least, of the coefficients $\alpha_1,\ldots ,\alpha_r$ is irrational, then the sequence 
$(f(n))_{n=1}^\infty $ is equidistributed modulo $1$. Although an analogue of this conclusion in the function field 
setting fails in general, in this paper we are able to derive a conclusion that is the best possible analogue involving 
an irrationality hypothesis on only one coefficient.\par

In order to describe our new conclusions, we require some notational infrastructure. Let $\mathbb F_q$ denote the 
finite field of $q$ elements whose characteristic is $p$, and suppose that $q=p^m$ throughout. Let 
$\mathbb K=\mathbb F_q(t)$ be the field of fractions of the associated polynomial ring $\mathbb F_q[t]$. Given 
$f/g\in \mathbb K$, with $f,g\in \mathbb F_q[t]$ and $g\ne 0$, we define the norm $|f/g|=q^{\deg f-\deg g}$, with 
the convention that $\deg 0=-\infty$. The completion of $\mathbb K$ with respect to this norm is 
$\mathbb K_\infty=\mathbb F_q((1/t))$, the field of formal Laurent series in $1/t$. Thus, the elements $\alpha$ of 
$\mathbb K_\infty$ can be written in the form $\alpha=\sum_{i=-\infty}^na_it^i$ for some $n\in \mathbb Z^+$ and 
$a_i\in \mathbb F_q$ $(i\le n)$. We may define $\text{ord}\, \alpha$ for $\alpha\in \mathbb K_\infty$ by putting 
$\text{ord}\, 0=-\infty$, and otherwise
\[
\text{ord}\biggl( \sum_{i=-\infty}^na_it^i\biggr) =\max \{ i\in \mathbb Z:a_i\ne 0\} .
\]
One then has the norm on $\mathbb K_\infty$, given by $|\alpha|=q^{\text{ord}\,\alpha}$, that coincides for 
$\alpha\in \mathbb K$ with the norm that we defined earlier. We note in passing that the coefficient $a_{-1}$ here is 
called the {\it residue} of $\alpha$, denoted by $\text{res}(\alpha)$. Thus, when $n\in \mathbb Z^+$, we have
\begin{equation}\label{1.1}
\text{res}\biggl( \sum_{i=-\infty}^na_it^i\biggr) =a_{-1}.
\end{equation}

\par As will be familiar to conversant readers, the function field analogues of $\mathbb Z$, $\mathbb Q$ and 
$\mathbb R$ are respectively $\mathbb F_q[t]$, $\mathbb K$ and $\mathbb K_\infty$. Let
\[
\mathbb T=\{ \alpha\in \mathbb K_\infty:\text{ord}\, \alpha<0\}.
\]
This compact subgroup of $\mathbb K_\infty$ is the analogue of the unit interval $[0,1)$ in the classical setting. 
When $\alpha=\sum_{i=-\infty}^na_it^i\in \mathbb K_\infty$ with $n\in \mathbb Z^+$, we define
\[
\{ \alpha\}=\sum_{i\le -1}a_it^i\in \mathbb T
\]
to be the {\it fractional part} of $\alpha$. Let $\mu$ be a normalized Haar measure on $\mathbb T$ such that 
$\mu(\mathbb T)=1$. A function field analogue of equidistribution modulo $1$ was introduced by Carlitz 
\cite[\S4]{Car1952}.
\begin{definition}\label{definition1.1}
Let $f$ be a $\mathbb K_\infty$-valued function defined over $\mathbb F_q[t]$. We say that 
$(f(u))_{u\in \mathbb F_q[t]}$ is {\it equidistributed in $\mathbb T$} if, for every open ball 
$\mathcal B\subset \mathbb T$, one has
\[
\lim_{N\rightarrow \infty}q^{-N}\text{card}\{ u\in \mathbb F_q[t]: \text{$\text{ord}\,u<N$ and $\{f(u)\}\in \mathcal B$}
\}=\mu(\mathcal B).
\]
\end{definition}

 An alternative definition, involving {\it cylinder sets}, is given in \cite[Definition~1.1]{LLW2025}. A moment of 
reflection will reveal that this definition is equivalent to Definition \ref{definition1.1} above.\par

\begin{definition}\label{definition1.2}
Let $f$ be a $\mathbb K_\infty$-valued function defined over $\mathbb F_q[t]$. When
$(f(mu+s))_{u\in \mathbb F_q[t]}$ is equidistributed in $\mathbb{T}$ for all $m,s\in \mathbb{F}_q[t]$ with $m\neq 0$, 
we say that $(f(u))_{u\in \mathbb{F}_q[t]}$ is {\it totally equidistributed in $\mathbb T$}.
\end{definition}

The notion of total equidistribution can be found, for instance, in \cite[Sec.~1.1.1]{Tao2012}.
It is immediate that total equidistribution is stronger than  equidistribution.\par

The problem of finding a proper function field analogue of Weyl's equidistribution theorem for polynomials is the 
main subject of this paper, and is a problem that was first considered by Carlitz \cite{Car1952} in 1952. We recall that 
an element $\alpha\in \mathbb K_\infty$ is called {\it irrational} when $\alpha\not\in \mathbb K$. In 
\cite[Theorem 12]{Car1952}, Carlitz shows that when $f\in \mathbb K_\infty[x]$ has degree less than $p$, and 
$f(x)-f(0)$ has an irrational coefficient, then $(f(u))_{u\in \mathbb F_q[t]}$ is equidistributed in $\mathbb T$. This 
conclusion was subsequently recovered by Dijksma \cite{Dij1970} with a stronger notion of equidistribution. As 
remarked by Carlitz himself, the condition $\deg f<p$ cannot be removed, as there exist uncountably many irrational 
elements $\alpha \in \mathbb K_\infty$ having the property that $(\alpha u^p)_{u\in \mathbb F_q[t]}$ is not 
equidistributed in $\mathbb T$ (see also the discussion of \cite[Example 1.2]{LLW2025}). However, the third, fourth 
and fifth authors have obtained equidistribution results that address many situations in which $\deg f\ge p$ (see 
\cite[Theorem 1.4 and Proposition 5.2]{LLW2025}). These conclusions make progress towards a conjecture made by 
these authors (see \cite[Conjecture 1.3]{LLW2025}) that, if true, would be the best possible result that could hold in 
which an irrationality hypothesis is imposed on a single coefficient. Our main goal in this paper is to resolve this 
conjecture in full.

\begin{theorem}\label{theorem1.2} Let $\mathcal K$ be a finite set of positive integers, suppose that 
$\alpha_r\in \mathbb K_\infty$ for $r\in \mathcal K\cup \{ 0\}$, and define
\[
f(x)=\sum_{r\in \mathcal K\cup \{0\}}\alpha_rx^r.
\]
Suppose that $\alpha_k$ is irrational for some $k\in \mathcal K$ satisfying $p\nmid k$ and furthermore 
$p^vk\not \in \mathcal K$ for any $v\in \mathbb Z^+$. Then the sequence $(f(u))_{u\in \mathbb F_q[t]}$ is totally
equidistributed in $\mathbb T$.
\end{theorem}

We note that, in the original version of this paper, our Theorem 1.3 was stated with the classical equidistribution 
conclusion in place of the total equidistribution in the current version. This stronger result was presented in a note of 
Ackelsberg and Bergelson \cite[Theorem 3.5]{AB2026} that appeared on the arXiv following the original version of 
this paper. Since the proof of the stronger statement is easily accomplished within our original proof, we have 
modified the statement of Theorem \ref{theorem1.2} to better align with \cite[Theorem 3.5]{AB2026}. See also 
Lemmata 3.3.2 and 3.3.3 of the first author's thesis \cite{Cha2026}, for closely related results.\par

A related though less definitive conclusion is obtained in \cite[Theorem 1.4]{LLW2025}. In order to describe this 
result, given a set of positive integers $\mathcal K$, we introduce the {\it shadow} of $\mathcal K$, defined to be the 
set
\begin{equation}\label{def:shadow}
\mathcal S(\mathcal K)=\left\{ j\in \mathbb Z^+:\text{$p\nmid \binom{r}{j}$ for some $r\in \mathcal K$}\right\}.
\end{equation}
Notice that $\mathcal K\subseteq \mathcal S(\mathcal K)$ for all sets $\mathcal K$ of positive integers. The 
conclusion of \cite[Theorem 1.4]{LLW2025} is similar to that of Theorem \ref{theorem1.2} above, save that the 
condition  $p^vk\not\in \mathcal K$ for any $v\in \mathbb Z^+$, in the latter, is replaced by the more stringent 
hypothesis $p^vk\not\in \mathcal S(\mathcal K)$ for any $v\in \mathbb Z^+$.\par

A conclusion slightly more general than that of \cite[Theorem 1.4]{LLW2025} is obtained in 
\cite[Proposition 5.2]{LLW2025}. This too falls short of establishing \cite[Conjecture 1.3]{LLW2025}, now established 
in Theorem \ref{theorem1.2} above. Consider again a set $\mathcal K$ of positive integers. In order to describe 
\cite[Proposition 5.2]{LLW2025}, we begin by defining the set
\[
\mathcal K^*=\{ k\in \mathcal K:\text{$p\nmid k$ and $p^vk\not \in \mathcal S(\mathcal K)$ for any 
$v\in \mathbb Z^+$}\}.
\]
We now put $\mathcal K_0=\mathcal K$, and inductively define for each $n\ge 1$ the set
\[
\mathcal K_n=\mathcal K_{n-1}\setminus \mathcal K_{n-1}^*.
\]
We then define the set of indices
\begin{equation}\label{1.2}
{\widetilde{\mathcal K}}=\bigcup_{n=0}^\infty\, \mathcal K_n^*.
\end{equation}
With this notation, the conclusion of \cite[Proposition 5.2]{LLW2025} shows that the conclusion of Theorem 
\ref{theorem1.2} holds whenever $\alpha_k$ is irrational for some $k\in {\widetilde{\mathcal K}}$. As we have noted 
already, this condition is again more stringent than that imposed in Theorem \ref{theorem1.2}.\par

We remark that, with little additional effort, the improvement of \cite[Theorem 1.4]{LLW2025} embodied in Theorem 
\ref{theorem1.2} can be imported also into the conclusions of \cite[Theorems 6.3 and 7.3]{LLW2025}. We leave to the 
reader the pedestrian details of confirming such enhancements.\par 

The conclusion of Theorem \ref{theorem1.2} has been claimed in recent work of Ackelsberg 
\cite[Theorem 1.2]{Ack2025}, based on earlier work of Bergelson and Leibman \cite[Theorem 0.3]{BL2016}. The latter 
work is based on a variant of van der Corput differencing. When working in rings of positive characteristic $p$, van 
der Corput differencing shares the same defect as Weyl differencing, in that the number of differencing steps taken 
is limited by this characteristic, potentially preventing the exponential sum under discussion from being bounded 
non-trivially when it contains monomials of degree larger than $p$. Although Bergelson and Leibman take care to 
work around the obstacles imposed by this well-known observation, there is an infelicity in the ``only if'' direction of 
\cite[Theorem 0.3]{BL2016}, for which we provide a counterexample in this note. This issue is addressed, with 
corrected statements, in Ackelsberg and Bergelson \cite[\S 3.1]{AB2026}, so that Ackelsberg's work \cite{Ack2025} is 
ultimately not impaired (see the subsequently published revision \cite{Ack2026}). Our own proof avoids certain 
difficulties associated with Weyl and van der Corput differencing by instead making use of ideas based on 
Vinogradov's mean value theorem and the large sieve inequality, as described in \cite{LLW2025}. We stress that the 
method presented here has advantages over that of Bergelson and Leibman. First, in contrast to the notion of 
well-distributed sequences in \cite{Ack2025} and \cite{BL2016}, which is defined in terms of F{\o}lner sequences, our 
method yields equidistribution for a broad class of sequences. For instance, 
$(f(x))_{x\equiv \pm 1 \,(\mathrm{mod}\, m)}$ and $(f(x))_{x\ \textup{irreducible}}$ fall outside the scope of 
Ackelsberg's \cite[Theorem 1.2]{Ack2025}. The former follows immediately from our Theorem~\ref{theorem1.2} and 
Ackelsberg and Bergelson \cite[Theorem 3.5]{AB2026}, while the equidistribution theory of the latter can also be 
treated by our method in ongoing work. Secondly, our method yields quantitative bounds for exponential sums, as 
stated in Theorem~\ref{theorem4.3}. Such estimates can be utilised in many other problems over $\mathbb{F}_q[t]$, 
including Diophantine approximation, Waring's problem, and additive combinatorics. See, for example, S.~Yamagishi 
\cite{Yam2016a, Yam2016b} and Z.~Ge \cite{Ge2026}.\par

Our paper is organised as follows. Having recalled some basic features of the harmonic analysis of function fields in 
\S2, paying close attention to the role of the Frobenius map, we explore implications for additive polynomials in \S3. 
This provides the groundwork for our proof of Theorem \ref{theorem1.2} in \S4. We conclude in \S5 by providing a 
counter-example to a claim in the conclusion of \cite[Theorem 1.2]{Ack2025} that we mentioned in the previous 
paragraph.\par 

The work presented in this paper was the result of a zoom collaboration initiated in mid-2023.

\section{Preliminaries} We begin our deliberations with an account of certain features of the analysis of function 
fields playing a role in our discussion of Theorem \ref{theorem1.2}. First, in order to describe an analogue of Weyl's 
criterion for equidistribution in the function field setting, we introduce some additional notation. Recall that we 
assume $q=p^m$. Let ${\rm{tr}}:\mathbb F_q\rightarrow \mathbb F_p$ denote the familiar trace map, given 
explicitly by
\[
{\rm{tr}}(a)=a+a^p+a^{p^2}+\ldots +a^{p^{m-1}}.
\]
We note that this map is additive, in the sense that ${\rm{tr}}(a+b)={\rm{tr}}(a)+{\rm {tr}}(b)$ whenever 
$a,b\in \mathbb F_q$. We next define the additive character $e_q:\mathbb F_q\rightarrow \mathbb C^\times$ by 
taking $e_q(a)=e^{2\pi {\rm i}{\rm{tr}}(a)/p}$. This character induces a map, which we denote by $e(\cdot)$, from 
$\mathbb K_\infty$ to $\mathbb C^\times$, defined for each $\alpha\in \mathbb K_\infty$ by putting
\[
e(\alpha)=e_q({\rm{res}}\, \alpha),
\]
where ${\rm{res}}\, \alpha$ is defined via \eqref{1.1}. One can check that 
$e:\mathbb K_\infty \rightarrow \mathbb C^\times$ is a non-trivial continuous additive character on 
$\mathbb K_\infty$. The following analogue of Weyl's criterion for equidistribution was introduced by Carlitz.

\begin{theorem}\label{theorem2.1} The sequence $(a_u)_{u\in \mathbb F_q[t]}\subset \mathbb K_\infty$ is 
equidistributed in $\mathbb  T$ if and only if for any $h\in \mathbb F_q[t]\setminus \{ 0\}$, we have
\[
\lim_{N\rightarrow \infty}q^{-N}\biggl| \sum_{\substack{u\in \mathbb F_q[t]\\ {\rm{ord}}\, u<N}}e(ha_u)\biggr|=0.
\]
\end{theorem}

\begin{proof} This is Carlitz \cite[Theorem 4]{Car1952}.
\end{proof}

We end this section by recording a fundamental property of continuous additive characters on $\mathbb K_\infty$. 
This is, in fact, a consequence of a more general result concerning the self-duality of local fields as topological 
groups.

\begin{theorem}\label{theorem2.2} Let $\chi:\mathbb K_\infty\rightarrow \mathbb C^\times$ be a continuous 
additive character on $\mathbb K_\infty$. Then there exists a unique $\eta\in \mathbb K_\infty$ such that, for all 
$\xi \in \mathbb K_\infty$, one has $\chi(\xi)=e(\eta \xi)$.
\end{theorem}

\begin{proof} This is an immediate consequence of the Corollary to Theorem 3 of 
Weil \cite[Section 5 of Chapter 2]{Wei1967}.
\end{proof}

In \cite{Car1952} and \cite[Example 1.2]{LLW2025}, it is explained that for certain irrational 
$\alpha\in \mathbb K_\infty$, the sequence $(\alpha u^p)_{u\in \mathbb F_q[t]}$ is not equidistributed in 
$\mathbb T$. This is achieved via an explicit computation. One may also explain this failure of equidistribution by 
explicitly computing the element $\eta \in \mathbb K_\infty$, whose existence is assured by Theorem 
\ref{theorem2.2}, having the property that for all $\xi\in \mathbb K_\infty$, one has $e(\alpha \xi^p)=e(\eta \xi)$. We 
record the relevant explicit construction in the form of a lemma. In this context, we note that since we assume that 
$q=p^m$, the Frobenius map $a\mapsto a^p$ on $\mathbb F_q$ possesses an inverse map $b\mapsto b^{1/p}$ 
given by putting $b^{1/p}=b^{p^{m-1}}$ for each $b\in \mathbb F_q$. Equipped with this inverse to Frobenius, we 
define the mapping $\psi:\mathbb K_\infty\rightarrow \mathbb K_\infty$ by putting
\begin{equation}\label{2.1}
\psi\biggl( \sum_{i\le n}a_it^i\biggr) =\sum_{pj+p-1\le n}a_{pj+p-1}^{1/p}t^j.
\end{equation}

\begin{lemma}\label{lemma2.3}
For all $\alpha,\xi\in \mathbb K_\infty$, one has $e(\alpha \xi^p)=e(\psi(\alpha)\xi)$.
\end{lemma}

\begin{proof} Given $\alpha,\xi\in \mathbb K_\infty$, we may write
\[
\alpha=\sum_{i\in \mathbb Z}a_it^i\quad \text{and}\quad \xi=\sum_{i\in \mathbb Z}b_it^i,
\]
where $a_i,b_i\in \mathbb F_q$ $(i\in \mathbb Z)$, and one has $a_i=b_i=0$ for $i$ sufficiently large. Equipped with 
this notation, one sees that
\begin{align}
\alpha \xi^p&=\biggl( \sum_{i\in \mathbb Z}a_it^i\biggr) \biggl( \sum_{j\in \mathbb Z}b_jt^j\biggr)^p\notag \\
&= \biggl( \sum_{i\in \mathbb Z}a_it^i\biggr) \biggl( \sum_{j\in \mathbb Z}b_j^pt^{jp}\biggr) \notag \\
&= \sum_{l\in \mathbb Z}\biggl( \sum_{j\in \mathbb Z}a_{l-jp}b_j^p\biggr) t^l.\label{2.2}
\end{align}
Meanwhile, similarly, one has
\begin{align}
\psi(\alpha)\xi &=\biggl( \sum_{i\in \mathbb Z}a_{pi+p-1}^{1/p}t^i\biggr) \biggl( \sum_{j\in \mathbb Z}b_jt^j\biggr)
\notag \\
&= \sum_{l\in \mathbb Z}\biggl( \sum_{j\in \mathbb Z}a_{p(l-j)+p-1}^{1/p}b_j\biggr) t^l.\label{2.3}
\end{align}

\par A comparison of the coefficients of $t^{-1}$ between the formulae \eqref{2.2} and \eqref{2.3} reveals that
\[
\left( {\rm{res}}(\psi(\alpha )\xi)\right)^p=\biggl( \sum_{j\in \mathbb Z}a_{-jp-1}^{1/p}b_j\biggr)^p=
\sum_{j\in \mathbb Z}
a_{-jp-1}b_j^p={\rm{res}(\alpha \xi^p)}.
\]
Since ${\rm{tr}}(c^p)={\rm{tr}}(c)$ for all $c\in \mathbb F_q$, we deduce that 
${\rm{tr}}({\rm{res}}(\psi(\alpha)\xi))={\rm{tr}}({\rm{res}}(\alpha \xi^p))$, whence $e(\psi(\alpha)\xi)=e(\alpha \xi^p)$. 
This completes the proof of the lemma.
\end{proof}

We now return briefly to discuss the relevance of this lemma for \cite[Example 1.2]{LLW2025}. In that example, one 
considers an irrational element $\alpha\in \mathbb K_\infty$ of the shape $\alpha=\sum_{i\le n}a_it^i$, with 
$a_{-1}=a_{-p-1}=\ldots =0$. One sees that the map $\psi$ defined in \eqref{2.1} has the property that
\[
\psi(\alpha)=\sum_{pj+p-1\le n}a_{pj+p-1}^{1/p}t^j\in \mathbb F_q[t].
\]
Thus, one sees that whenever $u\in \mathbb F_q[t]$, one has
\[
e(\alpha u^p)=e(\psi(\alpha)u)=1.
\]
In particular, it is evident from this viewpoint that $(\alpha u^p)_{u\in \mathbb F_q[t]}$ cannot be equidistributed in 
$\mathbb T$. This example is readily generalized. Let $k\in \mathbb N$. Then, in a similar manner, one sees that 
whenever $\alpha$ and $\beta$ are elements of $\mathbb K_\infty$ having the property that 
$\psi(\alpha)+\beta\in \mathbb F_q[t]$, then
\[
e(\alpha u^{kp}+\beta u^k)=e\bigl( (\psi(\alpha)+\beta)u^k\bigr)=1.
\]
Thus, there exist irrational elements $\alpha,\beta\in \mathbb K_\infty$ having the property that
\[
\bigl( \alpha u^{kp}+\beta u^k\bigr)_{u\in \mathbb F_q[t]}
\]
is not equidistributed in $\mathbb T$. In a certain sense, the monomials $u^{kp}$ and $u^k$ interfere with each 
other in characteristic $p$.

\section{Additive polynomials}
A polynomial $A\in \mathbb K_\infty[x]$ is said to be {\it additive} if the identity $A(x+y)=A(x)+A(y)$ holds in the 
polynomial ring $\mathbb K_\infty[x,y]$. One may verify that every additive polynomial assumes the form
\[
A(x)=\sum_{\nu=0}^n\alpha_\nu x^{p^\nu}
\]
for some non-negative integer $n$ and $\alpha_0,\alpha_1,\ldots ,\alpha_n\in \mathbb K_\infty$. We denote by 
$\mathcal A$ the set of all additive polynomials lying in $\mathbb K_\infty[x]$. The relevance of the set of additive 
polynomials for our present purposes is that every polynomial in $\mathbb K_\infty[x]$ may be written canonically in 
terms of additive polynomials.

\begin{lemma}\label{lemma3.1} Suppose that $f\in \mathbb K_\infty[x]$. Then there exists a unique finite set 
$\mathcal R$ of positive integers, each coprime to $p$, and a unique collection $(A_r)_{r\in \mathcal R}$ of non-zero 
additive polynomials, such that
\begin{equation}\label{3.1}
f(x)=f(0)+\sum_{r\in \mathcal R}A_r(x^r).
\end{equation}
\end{lemma}

\begin{proof} There exists a finite set $\mathcal I$ of positive integers, and a collection $(\alpha_i)_{i\in \mathcal I}$ 
of non-zero elements of $\mathbb K_\infty$, with the property that
\[
f(x)=f(0)+\sum_{i\in \mathcal I}\alpha_ix^i.
\]
As usual, we write $p^r\|m$ when $p^r|m$ and $p^{r+1}\nmid m$. Then, we define
\[
\mathcal R=\{ p^{-\nu}i:\text{$i\in \mathcal I$ and $p^\nu \|i$}\}.
\]
Thus, we have
\begin{align}
f(x)&=f(0)+\sum_{r\in \mathcal R}\sum_{\substack{\nu \in \mathbb Z^+\cup \{0\}\\ p^\nu r\in \mathcal I}}
\alpha_{p^\nu r}(x^r)^{p^\nu}\notag \\
&=f(0)+\sum_{r\in \mathcal R}A_r(x^r),\label{3.2}
\end{align}
where
\[
A_r(x)=\sum_{\substack{\nu \in \mathbb Z^+\cup\{0\}\\ p^\nu r\in \mathcal I}}\alpha_{p^\nu r}x^{p^\nu}.
\]
The conclusion of the lemma is immediate from the decomposition \eqref{3.2} on observing that $A_r(x)$ is an 
additive polynomial for each $r\in \mathcal R$.
\end{proof}

We next define a map $\tau: \mathcal A\rightarrow \mathbb K_\infty$ given explicitly by means of the formula
\begin{equation}\label{3.3}
\tau\biggl( \sum_{\nu=0}^n\alpha_\nu x^{p^\nu}\biggr) =\sum_{\nu=0}^n\psi^\nu(\alpha_\nu),
\end{equation}
in which $\psi$ is the function defined in \eqref{2.1}, and $\psi^\nu$ denotes the $\nu$-fold composition of $\psi$ 
with itself. Given an additive polynomial
\[
A(x)=\sum_{\nu=0}^n\alpha_\nu x^{p^\nu},
\]
we see from repeated application of Lemma \ref{lemma2.3} that this map has the property that, whenever 
$\xi\in \mathbb K_\infty$, one has
\begin{equation}\label{3.4}
e(A(\xi))=e\biggl( \sum_{\nu=0}^n\alpha_\nu \xi^{p^\nu}\biggr) =
e\biggl( \sum_{\nu=0}^n\psi^\nu(\alpha_\nu)\xi\biggr) =e(\tau(A)\xi). 
\end{equation}
This observation confirms explicitly the conclusion of Theorem \ref{theorem2.2} for the continuous additive character 
$\xi \mapsto e(A(\xi))$ on $\mathbb K_\infty$. We note in particular that \eqref{3.3} gives the relation 
$\tau(\alpha x)=\alpha$, for $\alpha\in \mathbb K_\infty$.\par

The map $\tau$ that we have introduced in \eqref{3.3} combines with the decomposition given in Lemma 
\ref{lemma3.1} to replace a given polynomial $f\in \mathbb K_\infty[x]$ by a substitute free of monomials of the 
shape $x^{pm}$ $(m\in \mathbb Z^+)$. This distillation provides a route by which one may avoid difficulties in 
applying the work of \cite{LLW2025} in investigations concerning equidistribution in $\mathbb T$.

\begin{lemma}\label{lemma3.2} Suppose that $f\in \mathbb K_\infty[x]$. Write $f(x)$ in the shape \eqref{3.1} for a 
suitable set of positive integers $\mathcal R$, each coprime to $p$, and a collection $(A_r)_{r\in \mathcal R}$ of non-
zero additive polynomials. Define
\begin{equation}\label{3.5}
g(x)=f(0)+\sum_{r\in \mathcal R}\tau(A_r)x^r.
\end{equation}
Then, for all $\xi\in \mathbb K_\infty$, one has $e(f(\xi))=e(g(\xi))$.
\end{lemma}

\begin{proof} On making use of \eqref{3.1}, \eqref{3.4} and \eqref{3.5}, one sees that
\[
e(f(\xi))=e(f(0))\prod_{r\in \mathcal R}e(A_r(\xi^r))=e(f(0))\prod_{r\in \mathcal R}e(\tau(A_r)\xi^r)=e(g(\xi)).
\]
This completes the proof of the lemma.
\end{proof}

\section{Equidistribution in $\mathbb T$}
Our goal in this section is to make use of the substitution principle embodied in Lemma \ref{lemma3.2} so as to 
apply the conclusions of \cite{LLW2025} in order to confirm Theorem \ref{theorem1.2}. We begin by recalling a 
consequence of \cite[Proposition 5.2]{LLW2025}.

\begin{theorem}\label{theorem4.1}
Let $\mathcal R$ be a finite set of positive integers coprime to $p$, and let $f\in \mathbb K_\infty[x]$ be a 
polynomial of the form
\[
f(x)=\alpha_0+\sum_{r\in \mathcal R}\alpha_rx^r,
\]
with $\alpha_i\in \mathbb K_\infty$ $(i\in \mathcal R\cup \{0\})$. Suppose that $\alpha_r$ is irrational for some index 
$r\in \mathcal R$. Then the sequence $(f(u))_{u\in \mathbb F_q[t]}$ is equidistributed in $\mathbb T$.
\end{theorem}

\begin{proof} It follows from the definition \eqref{1.2} that when $(r,p)=1$ for all $r\in \mathcal R$, one has 
${\widetilde{\mathcal R}}=\mathcal R$. In order to confirm this observation, note first that 
$\mathcal R_n\subset \mathcal R$ for each $n\ge 0$. Then, whenever $\mathcal R_n$ is non-empty, its largest 
element is coprime to $p$, and hence is also an element of $\mathcal R_n^*$. Since each non-empty set 
$\mathcal R_n^*$ contains at least one element not contained in $\mathcal R_{n+1}$, it follows that
\[
\bigcup_{n=0}^{{\rm{card}}(\mathcal R)}\mathcal R_n^*=\mathcal R,
\]
whence ${\widetilde{\mathcal R}}=\mathcal R$, as claimed. The conclusion of the theorem is therefore immediate 
from \cite[Proposition 5.2]{LLW2025}.
\end{proof}

We note that a simplified proof of Theorem \ref{theorem4.1} is available in which the set ${\widetilde{\mathcal R}}$, 
as well as the application of \cite[Proposition 5.2]{LLW2025}, is avoided. Indeed, since $\mathcal R$ contains only 
positive integers coprime to $p$, the conclusion of \cite[Theorem 4.1]{LLW2025} remains valid for $\mathcal R$ if, in 
the proof, we replace the shadow partial ordering on $\mathcal R^*$ with the standard integer ordering on 
$\mathcal R$. Consequently, the conclusion of \cite[Lemma 5.1]{LLW2025} also holds when $f$ has its coefficients 
$\alpha_r$ supported on $r\in \mathcal R\cup \{ 0\}$. One now obtains a proof of Theorem \ref{theorem4.1} by 
following the corresponding proof of \cite[Theorem 1.4]{LLW2025}.\par

This conclusion is applicable to the substitute polynomials generated by Lemma \ref{lemma3.2}.

\begin{theorem}\label{theorem4.2} Let $f\in \mathbb K_\infty[x]$, and write $f(x)$ in the shape \eqref{3.1} for a 
suitable set of positive integers $\mathcal R$, each coprime to $p$, and a collection $(A_r)_{r\in \mathcal R}$ of 
non-zero additive polynomials. Suppose that, for all $h\in \mathbb F_q[t]\setminus \{0\}$, there exists 
$r\in \mathcal R$ having the property that $\tau(hA_r)$ is irrational. Then the sequence $(f(u))_{u\in \mathbb F_q[t]}$ 
is totally equidistributed in $\mathbb T$.
\end{theorem}

In the original version of our paper, we recorded a Lemma 4.2 equivalent to Theorem \ref{theorem4.2} above save 
that the total equidistribution conclusion here was replaced by the classical equidistribution conclusion therein. As 
we have remarked already in the introduction, this stronger conclusion appeared as Theorem 3.5 in a note of 
Ackelsberg and Bergelson \cite{AB2026} that was uploaded to the arXiv following the original version of this paper.

Suppose that $f(x)=\sum_{i\in \mathcal{K}\cup\{0\}}\alpha_i x^i\in \mathbb{K}_\infty[x]$, where 
$\mathcal{K}\subseteq\mathbb{Z}^+$ is finite, and suppose further that there exists $k\in \mathcal{K}$ with 
$p\nmid k$ such that $\alpha_k$ is irrational and $p^v k\notin \mathcal{K}$ for any $v\in \mathbb{Z}^+$. By the 
discussion following \eqref{3.3}, the coefficient
\[
\tau(hA_k)=\tau(h\alpha_k x)=h\alpha_k
\]
is irrational for every $h\in \mathbb F_q[t]\setminus \{0\}$. Thus Theorem \ref{theorem1.2} is an immediate corollary 
of Theorem \ref{theorem4.2}. We now turn to the proof of Theorem \ref{theorem4.2}.

\begin{proof}[Proof of Theorem \ref{theorem4.2}]
Fix an arbitrary element $h\in \mathbb F_q[t]\setminus \{ 0\}$. The hypotheses of Theorem \ref{theorem4.2} ensure 
that
\[
hf(x)=hf(0)+\sum_{r\in \mathcal R}hA_r(x^r),
\]
where $\mathcal R$ is a set of positive integers coprime to $p$, and $(hA_r)_{r\in \mathcal R}$ is a collection of 
nonzero  additive polynomials. By Lemma \ref{lemma3.2}, the 
polynomial
\[
g_h(x)=hf(0)+\sum_{r\in \mathcal R}\tau(hA_r)x^r
\]
satisfies the property that, for all $u\in \mathbb F_q[t]$, one has
\[
e(hf(u))=e(g_h(u)).
\]
Furthermore, the hypothesis of Theorem \ref{theorem4.2} implies that that polynomial $g_h(x)$ has at least one 
irrational coefficient $\tau(hA_r)$ for $r\in \mathcal R$. We fix $r_0$ to be the largest index $r\in \mathcal R$ having 
the property that $\tau(hA_r)$ is irrational. It then follows that $\tau(hA_r)$ is rational whenever $r\in \mathcal R$ 
and $r>r_0$. \par

We now refine the preceding argument to obtain total equidistribution. For $m,s\in \mathbb{F}_q[t]$ with 
$m\neq 0$, write
\[
g_h(mx+s)=\beta_0+ \sum_{i\in \mathcal{S}(\mathcal{R})}
\beta_i\, x^i,
\]
with
\[
\beta_i=\sum_{r\geq i}\binom{r}{i}\tau(hA_r)m^i s^{r-i},
\]
where $\mathcal{S}(\mathcal{R})$ is the shadow of $\mathcal{R}$ defined in \eqref{def:shadow}. We note that 
$r_0\in \mathcal{S}(\mathcal{R})$ and $\beta_{r_0}$ is irrational, and further that $\beta_i$ is rational for all $i>r_0$. 
We now find from Lemma \ref{lemma3.1} that $g_h(mx+s)$ can be written in the form
\[
g_h(mx+s)=\beta_0 + \sum_{j\in \mathcal{J}} B_j(x^j),
\]
where
\[
\mathcal{J}=\{p^{-\nu}i: i\in \mathcal{S}(\mathcal{R})\textup{ and }p^\nu\| i\}
\]
and
\[
B_j(x)=\sum_{\substack{\nu\in \mathbb{Z}^+\cup\{0\}\\ p^\nu j\in \mathcal{S}(\mathcal{R})}}\beta_{p^\nu j} 
x^{p^{\nu}}\quad (j\in \mathcal{J}).
\]
Define
\[
g_{h,m,s}(x)=\beta_0 + \sum_{j\in \mathcal{J}}\tau(B_j) x^j.
\]
Then
\[
e(hf(mu+s))=e(g_h(mu+s))=e(g_{h,m,s}(u))
\]
for all $h,m\in \mathbb{F}_q[t]\setminus\{0\}$ and $s\in \mathbb{F}_q[t]$. 

From the definition of $\tau(\cdot)$ in \eqref{3.3}, we observe that
\[
\tau(B_{r_0})=\tau\Bigg(\sum_{\substack{\nu\in \mathbb{Z}^+\cup\{0\}\\ p^\nu r_0\in \mathcal{S}(\mathcal{R})}}
\beta_{p^\nu r_0} x^{p^{\nu}}\Bigg)
=
\beta_{r_0}+
\sum_{\substack{\nu\in \mathbb{Z}^+\\ p^{\nu}r_0\in \mathcal{S}(\mathcal{R})}}
\psi^{\nu}(\beta_{p^{\nu}r_0}).
\]
Recalling the $\psi$-function defined in \eqref{2.1}, we note that $\psi(\alpha)$ is rational whenever 
$\alpha\in \mathbb{K}_\infty$ is rational. This is because $\alpha$ is rational if and only if the coefficients of 
successive powers of $t^{-1}$ in $\alpha$ are eventually periodic. Since $\beta_{r_0}$ is irrational and $\beta_i$ is 
rational for all $i>r_0$, it follows that $\tau(B_{r_0})$ is irrational.\par

Finally, since $g_{h,m,s}(x)$ satisfies the hypotheses of Theorem~\ref{theorem4.1}, we find that the sequence 
$(g_{h,m,s}(u))_{u\in \mathbb F_q[t]}$ is equidistributed in $\mathbb T$. We therefore infer from 
Theorem~\ref{theorem2.1} that
\[
\lim_{N\rightarrow \infty}q^{-N}\biggl| \sum_{\substack{u\in \mathbb F_q[t]\\ {\rm ord}\, u<N}}e(hf(mu+s))\biggr| =
\lim_{N\rightarrow \infty}q^{-N}\biggl| \sum_{\substack{u\in \mathbb F_q[t]\\ {\rm ord}\, u<N}}e(g_{h,m,s}(u))\biggr| 
=0.
\]
Since $h,m,s$ are arbitrary, a further application of Theorem~\ref{theorem2.1} yields the conclusion that 
$(f(u))_{u\in \mathbb F_q[t]}$ is totally equidistributed in $\mathbb T$. This completes the proof of Theorem 
\ref{theorem4.2}.
\end{proof}

The argument that we applied in the proof of Theorem \ref{theorem1.2} can be incorporated into the method of 
proof of \cite[Theorem 4.1]{LLW2025} so as to deliver the following refinement of \cite[Proposition 4.2]{LLW2025}. 
This constitutes the quantitative bound to which we alluded in the introduction. Since the proof of this result 
introduces no new ideas beyond those already introduced in our proof of Theorem \ref{theorem1.2} and that of 
\cite[Theorem 4.1]{LLW2025}, we leave the details as an exercise for the reader.

\begin{theorem} \label{theorem4.3}
Fix $q$ and a finite set $\mathcal K \subset \mathbb Z^+$. There exist positive constants $c$ and $C$, depending 
only on $\mathcal K$ and $q$, such that the following holds. Let $\varepsilon>0$ and let $N$ be sufficiently large in 
terms of $\mathcal K$, $\varepsilon$ and $q$. Suppose that $f(x)=\sum_{r\in \mathcal K\cup\{0\}}\alpha_r x^{r}$ is a 
polynomial with coefficients in $\mathbb K_\infty$ satisfying the bound
\[
\biggl| \sum_{\substack{u\in \mathbb F_q[t]\\ {\rm{ord}}\, u<N}} e(f(u))\biggr| \geq q^{N-\eta },
\]
for some positive number $\eta$ with $\eta\le cN$. Then, for any $k\in \mathcal K$ satisfying $p\nmid k$ and 
furthermore $p^vk\not\in \mathcal K$ for any $v\in \mathbb Z^+$, there exist $a_k\in \mathbb F_q[t]$ and monic 
$g_k\in \mathbb F_q[t]$ such that
\[
{\rm{ord}} (g_k\alpha_k-a_k)<-kN+\varepsilon N+C\eta \quad \textrm{and} \quad {\rm{ord}}\,g_k\leq 
\varepsilon N+C\eta.
\]
\end{theorem}

\section{An example related to work of Bergelson and Leibman}
We complete our discussion of equidistribution in $\mathbb T$ by exhibiting an example that provides evidence for 
the infelicity of the ``only if'' direction in the work of Bergelson and Leibman \cite{BL2016}. Since this example 
illuminates certain aspects of the theory of equidistribution in positive characteristic, we provide a largely 
self-contained account. We formulate the relevant claimed corollary in the form of a hypothesis.

\begin{hypothesis}\label{hypothesis5.1}
Suppose that $\mathcal R$ is a set of positive integers, each coprime to $p$, let $(A_r)_{r\in \mathcal R}$ be a 
collection of additive polynomials, and define
\[
f(x)=\sum_{r\in \mathcal R}A_r(x^r).
\]
If, for some $r\in \mathcal R$, the sequence $(A_r(u))_{u\in \mathbb F_q[t]}$ is equidistributed in $\mathbb T$, then 
$(f(u))_{u\in \mathbb F_q[t]}$ is also equidistributed in $\mathbb T$.
\end{hypothesis}

This assertion is an immediate consequence of the final conclusion of Ackelsberg \cite[Theorem 2.1]{Ack2025}, which 
he attributes to \cite[Theorem 0.3]{BL2016}. In the language of Ackelsberg, the notion of equidistribution in 
$\mathbb T$ is replaced by the notion of being well-distributed in $\mathbb T$. However, if 
$(A_r(u))_{u\in \mathbb F_q[t]}$ is equidistributed in $\mathbb T$ for an additive polynomial $A_r(x)$, then it is 
well-distributed in $\mathbb T$, and the hypothesis preliminary to the conclusion of Hypothesis \ref{hypothesis5.1} 
suffices (according to \cite[Theorem 2.1]{Ack2025}) to ensure that $(f(u))_{u\in \mathbb F_q[t]}$ is well-distributed in 
$\mathbb T$, and hence also equidistributed in $\mathbb T$. We shall exhibit an example demonstrating that 
Hypothesis \ref{hypothesis5.1} may fail.\par

We begin by recording a trivial observation that streamlines subsequent aspects of our discussion.

\begin{lemma}\label{lemma5.2}
Suppose that $A(x)\in \mathbb K_\infty[x]$ is an additive polynomial. Then the following are equivalent:
\begin{itemize}
\item[(a)] for all $h\in \mathbb F_q[t]\setminus \{0\}$, the character $x\mapsto e(hA(x))$ is non-trivial on 
$\mathbb F_q[t]$;
\item[(b)] for all $h\in \mathbb F_q[t]\setminus \{0\}$, one has $\tau(hA)\not\in \mathbb F_q[t]$;
\item[(c)] the sequence $(A(u))_{u\in \mathbb F_q[t]}$ is equidistributed in $\mathbb T$;
\item[(d)] the sequence $(A(u))_{u\in \mathbb F_q[t]}$ is well-distributed in $\mathbb T$.
\end{itemize}
\end{lemma}

\begin{proof}
In view of the relation \eqref{3.4}, the character $u\mapsto e(hA(u))$ is non-trivial on $\mathbb F_q[t]$ if and only if 
the character $u\mapsto e(\tau(hA)u)$ is likewise non-trivial on $\mathbb F_q[t]$. But the latter holds if and only if 
$\tau(hA)\not\in \mathbb F_q[t]$. Thus we find that (a) holds if and only if (b) holds.\par

Next, we observe that from Theorem \ref{theorem2.1} it follows that $(A(u))_{u\in \mathbb F_q[t]}$ is equidistributed 
in $\mathbb T$ if and only if, for any $h\in \mathbb F_q[t]\setminus \{ 0\}$, we have
\begin{equation}\label{5.1}
\lim_{N\rightarrow \infty}q^{-N}\biggl| \sum_{\substack{u\in \mathbb F_q[t]\\ {\rm{ord}}\, u<N}}e(hA(u))\biggr|=0.
\end{equation}
Since $e(hA(u))=e(\tau(hA)u)$, we see that when $\tau(hA)\in \mathbb F_q[t]$, one has $e(hA(u))=1$. Meanwhile, 
when $\tau(hA)\not\in \mathbb F_q[t]$, it follows from Kubota \cite[Lemma 7]{Kub1974} that
\[
\sum_{\substack{u\in \mathbb F_q[t]\\ {\rm{ord}}\,u<N}}e(hA(u))=\begin{cases} q^N,&
\text{when ${\rm{ord}}\{ \tau(hA)\}<-N$},
\\
0,&\text{when ${\rm{ord}}\{ \tau(hA)\}\ge -N$}.\end{cases}
\]
Thus, provided that $\{ \tau(hA)\}\ne 0$, we see that for large enough values of $N$ one has 
${\rm{ord}}\{ \tau(hA)\}\ge -N$, and hence the relation \eqref{5.1} must hold. We have therefore shown that 
$(A(u))_{u\in \mathbb F_q[t]}$ is equidistributed in $\mathbb T$ if and only if, for any 
$h\in \mathbb F_q[t]\setminus \{0\}$, one has $\tau(hA)\not\in \mathbb F_q[t]$. Thus, we conclude that (c) holds if 
and only if (b) holds.\par

Finally, it is a trivial consequence of the definition of well-distribution employed in \cite{Ack2025} that 
$(A(u))_{u\in \mathbb F_q[t]}$ is well-distributed in $\mathbb T$ if and only if it is equidistributed in $\mathbb T$. 
Indeed, assume that the latter holds. Then, since the additive property of $A$ ensures that $A(u+a)=A(u)+A(a)$ for 
each $a\in \mathbb F_q[t]$, the sequence $(A(u+a))_{u\in \mathbb F_q[t]}$ is equidistributed in $\mathbb T$, 
uniformly in $a$. It follows that $A:\mathbb F_q[t]\rightarrow \mathbb K_\infty$ is equidistributed in $\mathbb T$ 
along any F{\o}lner sequence, and hence $(A(u))_{u\in \mathbb F_q[t]}$ is well-distributed in $\mathbb T$. The 
reverse implication is trivial. Thus, we conclude that (d) holds if and only if (c) holds.\par

On combining these conclusions, we complete the proof of the lemma.
\end{proof}

Our counter-example to Hypothesis \ref{hypothesis5.1} rests on the following result concerning additive 
polynomials.

\begin{theorem}\label{theorem5.3}
Suppose that $\gamma\in \mathbb K_\infty \setminus \mathbb F_q[t]$. Then there exists an additive polynomial 
$A(x)\in \mathbb K_\infty [x]$ having the properties:\begin{itemize}
\item[(a)] one has $\tau(A)=\gamma$;
\item[(b)] for all $h\in \mathbb F_q[t]\setminus \{0,1\}$, one has $\tau(hA)\not\in \mathbb F_q[t]$.
\end{itemize}
\end{theorem}

We defer the proof of this conclusion for the time-being, and first concentrate on applying this result to obtain a 
counter-example to Hypothesis \ref{hypothesis5.1}.\medskip

\noindent{\bf Counter-example to Hypothesis \ref{hypothesis5.1}.} We begin by observing that the conclusion of 
Theorem \ref{theorem5.3} shows that there exists an additive polynomial $A(x)$ having the property that 
$\tau(A)=1/t$, and satisfying the condition that
\[
\tau(hA)\not\in \mathbb F_q[t]\quad \text{for all}\quad h\in \mathbb F_q[t]\setminus \{0,1\}.
\]
In particular, it then follows from Lemma \ref{lemma5.2} that the sequence $(A(u))_{u\in \mathbb F_q[t]}$ is 
equidistributed in $\mathbb  T$. When ${\rm{char}}(\mathbb F_q)\ne 2$, our counter-example to Hypothesis 
\ref{hypothesis5.1} is given by the polynomial
\begin{equation}\label{5.2}
f(x)=A(x^{q+1})-x^2/t.
\end{equation} 

\par In order to confirm that the polynomial $f$ defined in \eqref{5.2} does indeed furnish a counter-example to 
Hypothesis \ref{hypothesis5.1}, put
\[
\mathcal R=\{q+1,2\},\quad A_{q+1}(x)=A(x),\quad \text{and}\quad A_2(x)=-x/t.
\]
We have seen that the sequence $(A_{q+1}(u))_{u\in \mathbb F_q[t]}$ is equidistributed in $\mathbb T$, so 
Hypothesis \ref{hypothesis5.1} implies that $(f(u))_{u\in \mathbb F_q[t]}$ is also equidistributed in $\mathbb T$. 
However, on making use of \eqref{3.4}, we find that
\[
e(f(u))=e\Bigl(A(u^{q+1})-\frac{u^2}{t}\Bigr)=e\Bigl( \tau(A)u^{q+1}-\frac{u^2}{t}\Bigr)=e\Bigl( \frac{u^{q+1}}{t}-
\frac{u^2}{t}\Bigr) .
\]
But for all $u\in \mathbb F_q[t]$, one has $u^q-u\equiv 0\mmod{t}$, and hence
\[
e(f(u))=e\Bigl( \frac{u(u^q-u)}{t}\Bigr) =1.
\]
Consequently, the sequence $(f(u))_{u\in \mathbb F_q[t]}$ cannot be equidistributed in $\mathbb T$, since it fails the 
analogue of Weyl's criterion exhibited in Theorem \ref{theorem2.1}. This conclusion contradicts our earlier deduction 
that Hypothesis \ref{hypothesis5.1} implies that $(f(u))_{u\in \mathbb F_q[t]}$ is equidistributed in $\mathbb T$. We 
are therefore forced to infer that Hypothesis \ref{hypothesis5.1} fails for the example \eqref{5.2}.\qed
\medskip

The reader may care to check that the polynomial $f$ defined in \eqref{5.2} also furnishes a counter-example to 
Hypothesis \ref{hypothesis5.1} when ${\rm{char}}(\mathbb F_q)=2$. In this scenario, we put
\[
\mathcal R=\{q+1,1\},\quad A_{q+1}(x)=A(x),\quad \text{and}\quad A_1(x)=x^2/t.
\]
The argument applied above in the case  ${\rm{char}}(\mathbb F_q)\ne 2$ now shows, mutatis mutandis, that 
$(f(u))_{u\in \mathbb F_q[t]}$ is not equidistributed in $\mathbb T$, in contradiction with the earlier consequence of 
Hypothesis \ref{hypothesis5.1}.\par

We remark that our proof of Theorem \ref{theorem5.3} shows, in fact, that $A(x)$ may be chosen to have the shape 
$\alpha x+\beta x^p$. Moreover, without much difficulty, it would be possible to construct such an example with 
both $\alpha$ and $\beta$ irrational. This example then leads to a violation of Hypothesis \ref{hypothesis5.1} for the 
polynomial
\[
f(x)=\alpha x^{q+1}+\beta x^{p(q+1)}-x^2/t,
\]
as we have shown. Moreover, as must be the case, this polynomial $f(x)$ fails to be accessible to our Theorem 
\ref{theorem1.2}, as the reader may care to verify.
\medskip

We devote the remainder of this section to the proof of Theorem \ref{theorem5.3}. The difficulty that we must 
address revolves around the need to control the value of $\tau(hA)$, while at the same time fixing 
$\tau (A)=\gamma$, for a specified $\gamma\in \mathbb K_\infty \setminus \mathbb F_q[t]$. In order to assist in 
achieving this control, our argument will employ the mappings 
$\psi_l:\mathbb K_\infty\rightarrow \mathbb K_\infty$, defined for $l\in \mathbb Z$ by putting
\begin{equation}\label{5.3}
\psi_l\biggl( \sum_{i\le n}a_it^i\biggr)=\sum_{pj+l\le n}a_{pj+l}^{1/p}t^j,
\end{equation}
in which the map $b\mapsto b^{1/p}$ is again that described in the preamble to the statement of Lemma 
\ref{lemma2.3}. Note that, in this notation, our earlier mapping $\psi:\mathbb K_\infty\rightarrow \mathbb K_\infty$ 
is simply $\psi_{p-1}$. We now introduce the map $\Psi:\mathbb K_\infty \rightarrow \mathbb K_\infty^p$ by 
putting
\begin{equation}\label{5.4}
\Psi(\alpha)=\bigl( \psi_{p-1}(\alpha),\psi_{p-2}(\alpha),\ldots ,\psi_0(\alpha)\bigr) .
\end{equation}
This map is in fact equivalent to the {\it splitting isomorphism} introduced by Bergelson and Leibman 
\cite[page 934]{BL2016}, labelled as $\psi_1$ therein. It is evident from the construction of $\Psi$ that it defines a 
bijection from $\mathbb K_\infty$ to $\mathbb K_\infty^p$. The components of $\Psi$ satisfy the following 
multiplicative property.

\begin{lemma}\label{lemma5.4}
Suppose that $\alpha,\beta\in \mathbb K_\infty$ and $l\in \{ 0,1,\ldots ,p-1\}$. Then one has
\begin{equation}\label{5.5}
\psi_l(\alpha\beta)=\sum_{k=0}^{p-1}\psi_k(\alpha)\psi_{l-k}(\beta).
\end{equation}
\end{lemma}

\begin{proof} Given $\alpha,\beta\in \mathbb K_\infty$, we may write
\[
\alpha=\sum_{i\in \mathbb Z}a_it^i\quad \text{and}\quad \beta=\sum_{i\in \mathbb Z}b_it^i,
\]
where $a_i,b_i\in \mathbb F_q$ $(i\in \mathbb Z)$, and one has $a_i=b_i=0$ for $i$ sufficiently large. In this 
notation, one sees that
\[
\alpha \beta=\sum_{i\in \mathbb Z}\biggl( \sum_{r\in \mathbb Z}a_rb_{i-r}\biggr) t^i,
\]
whence
\begin{align}
\psi_l(\alpha \beta)&=\sum_{j\in \mathbb Z}\biggl( \sum_{r\in \mathbb Z}a_rb_{pj+l-r}\biggr)^{1/p}t^j\notag \\
&=\sum_{j\in \mathbb Z}\biggl( \sum_{k=0}^{p-1}\sum_{m\in \mathbb Z}a_{pm+k}b_{p(j-m)+l-k}\biggr)^{1/p}t^j.
\label{5.6}
\end{align}

\par Observe next that for each $a,b\in \mathbb F_q$, one has
\begin{equation}\label{5.7}
(a+b)^{1/p}=(a+b)^{p^{m-1}}=a^{p^{m-1}}+b^{p^{m-1}}=a^{1/p}+b^{1/p}.
\end{equation}
Consequently, for each $k\in \{0,1,\ldots ,p-1\}$, one has
\begin{align*}
\psi_k(\alpha)\psi_{l-k}(\beta)&=\biggl( \sum_{m\in \mathbb Z}a_{pm+k}^{1/p}t^m\biggr) 
\biggl( \sum_{n\in \mathbb Z}
b_{pn+l-k}^{1/p}t^n\biggr)\\
&=\sum_{j\in \mathbb Z} \biggl( \sum_{m\in \mathbb Z}a_{pm+k}b_{p(j-m)+l-k}\biggr)^{1/p}t^j.
\end{align*}
We therefore conclude that
\[
\sum_{k=0}^{p-1}\psi_k(\alpha)\psi_{l-k}(\beta)=\sum_{j\in \mathbb Z}\sum_{k=0}^{p-1}
\biggl( \sum_{m\in \mathbb Z}a_{pm+k}b_{p(j-m)+l-k}\biggr)^{1/p}t^j.
\]
An additional appeal to \eqref{5.7} therefore leads via \eqref{5.6} to the relation \eqref{5.5}. This completes the proof 
of the lemma.
\end{proof}

We apply this conclusion to the specific additive polynomial $A(x)=\alpha x+\beta x^p$, with 
$\alpha,\beta\in \mathbb K_\infty$. We observe that in these circumstances, one finds from \eqref{2.1} and 
\eqref{3.3} that
\[
\tau(hA)=\tau(h\alpha x+h\beta x^p)=\psi^0(h\alpha)+\psi^1(h\beta)=h\alpha+\psi_{p-1}(h\beta).
\]
Thus we infer that
\begin{equation}\label{5.8}
\tau(hA)=h\alpha+\psi_0(h)\psi_{p-1}(\beta)+\ldots +\psi_{p-1}(h)\psi_0(\beta).
\end{equation}
Our goal is to ensure that $\tau(hA)=\gamma$ when $h=1$, and when $h\in \mathbb F_q[t]\setminus \{0,1\}$, that 
$\tau(hA)$ avoids lying in $\mathbb F_q[t]$. Provided that the coefficients 
$\alpha, \psi_0(\beta),\ldots ,\psi_{p-1}(\beta)$ are sufficiently independent, this goal can be achieved with an 
investigation of the values of $h,\psi_0(h),\ldots ,\psi_{p-1}(h)$. Our discussion is eased by the introduction of the 
linear form
\[
\lambda(h;\alpha,\boldsymbol \xi)=h\alpha +\psi_0(h)\xi_0+\ldots +\psi_{p-1}(h)\xi_{p-1}.
\]

\begin{lemma}\label{lemma5.5} Suppose that $\gamma\in \mathbb K_\infty \setminus \mathbb F_q[t]$. Then there 
exist $\alpha,\xi_0,\ldots,\xi_{p-1}\in \mathbb K_\infty$ having the properties:
\begin{itemize}
\item[(a)] one has $\lambda(1;\alpha,\boldsymbol \xi)=\gamma$;
\item[(b)] whenever $h\in \mathbb F_q[t]\setminus \{0,1\}$, one has 
$\lambda (h;\alpha,\boldsymbol \xi)\not\in \mathbb F_q[t]$.
\end{itemize}
\end{lemma}

\begin{proof} Fix $\gamma\in \mathbb K_\infty \setminus \mathbb F_q[t]$. We are at liberty to choose 
$\xi_0,\ldots ,\xi_{p-1}\in \mathbb K_\infty$ in such a manner that $1,\xi_0,\ldots ,\xi_{p-1}$ are linearly independent 
over $\mathbb F_q(t)$. If $\gamma\not\in \mathbb F_q(t)$, moreover, we may insist on the stronger condition that 
$1,\gamma,\xi_0,\ldots ,\xi_{p-1}$ are linearly independent over $\mathbb F_q(t)$. In either case, we put 
$\alpha=\gamma-\xi_0$.\par

From the definition \eqref{5.3}, we have $\psi_0(1)=1$ and $\psi_1(1)=\ldots =\psi_{p-1}(1)=0$. Thus, we see that
\[
\lambda(1;\alpha,\boldsymbol \xi)=\alpha+\xi_0=\gamma.
\]
This confirms the first assertion of the lemma.\par

We now seek to establish the second assertion of the lemma. By way of seeking a contradiction, suppose, if possible, 
that there exists $h\in \mathbb F_q[t]\setminus \{0,1\}$ for which 
$\lambda (h;\alpha,\boldsymbol \xi)\in \mathbb F_q[t]$. Then there exists $b\in \mathbb F_q[t]$ having the property 
that
\begin{equation}\label{5.9}
h\alpha+\psi_0(h)\xi_0+\ldots +\psi_{p-1}(h)\xi_{p-1}=b.
\end{equation}

\par Suppose first that $\gamma\in \mathbb F_q(t)$. Since $\alpha=\gamma-\xi_0$, we find from \eqref{5.9} that
\[
(h\gamma-b)\cdot 1+(\psi_0(h)-h)\xi_0+\psi_1(h)\xi_1+\ldots +\psi_{p-1}(h)\xi_{p-1}=0.
\]
Each coefficient of $1,\xi_0,\ldots ,\xi_{p-1}$ here is an element of $\mathbb F_q(t)$, and hence it follows from the 
linear independence of $1,\xi_0,\ldots ,\xi_{p-1}$ over $\mathbb F_q(t)$ that
\[
\gamma=b/h,\quad \psi_0(h)=h,\quad \text{and}\quad \psi_1(h)=\ldots =\psi_{p-1}(h)=0.
\]
However, an inspection of the relation \eqref{5.3} ensures that when $h\in \mathbb F_q[t]$, one has the upper 
bound ${\rm{deg}}\, \psi_0(h)\le ({\rm{deg}}\, h)/p$. Thus, the relation $\psi_0(h)=h$ ensures that ${\rm{deg}}\, h=0$, 
whence $h\in \mathbb F_q$. Consequently, we have $\gamma=b/h\in \mathbb F_q[t]$, contradicting our hypothesis 
$\gamma\in \mathbb K_\infty \setminus \mathbb F_q[t]$.\par

In the alternative case $\gamma \not\in \mathbb F_q(t)$, we interpret the relation \eqref{5.9} in the form
\[
-b\cdot 1+h\gamma+(\psi_0(h)-h)\xi_0+\psi_1(h)\xi_1+\ldots +\psi_{p-1}(h)\xi_{p-1}=0.
\]
In this case, the linear independence of $1,\gamma,\xi_0,\ldots ,\xi_{p-1}$ over $\mathbb F_q(t)$ forces us to 
conclude that $b=h=0$. This again delivers a contradiction, and so the proof of the lemma is now complete.
\end{proof}

We are now equipped to complete the proof of Theorem \ref{theorem5.3}.

\begin{proof}[The proof of Theorem \ref{theorem5.3}] Suppose that 
$\gamma\in \mathbb K_\infty\setminus \mathbb F_q[t]$. The conclusion of Lemma \ref{lemma5.5} shows that there 
exist $\alpha, \xi_0,\ldots ,\xi_{p-1}\in \mathbb K_\infty$ having the property that
\[
\lambda(1;\alpha,\boldsymbol \xi)=\alpha+\psi_0(1)\xi_0+\ldots +\psi_{p-1}(1)\xi_{p-1}=\gamma,
\]
and further that, when $h\in \mathbb F_q[t]\setminus \{0,1\}$, one has
\[
h\alpha+\psi_0(h)\xi_0+\ldots +\psi_{p-1}(h)\xi_{p-1}\not\in \mathbb F_q[t].
\]
Fix any such choice of $\alpha,\xi_0,\ldots ,\xi_{p-1}$. Since the map 
$\Psi:\mathbb K_\infty \rightarrow \mathbb K_\infty^p$ defined via \eqref{5.4} is bijective, there exists 
$\beta\in \mathbb K_\infty$ such that
\[
\psi_{p-1-l}(\beta)=\xi_l\quad (0\le l\le p-1).
\]  
We define as our additive polynomial $A(x)=\alpha x+\beta x^p$. In view of \eqref{5.8}, we then have
\[
\tau(A)=\lambda(1;\alpha,\boldsymbol \xi)=\gamma ,
\]
and, furthermore, when $h\in \mathbb F_q[t]\setminus \{0,1\}$, one has
\[
\tau(hA)=\lambda(h;\alpha,\boldsymbol \xi)\not\in \mathbb F_q[t].
\]
This confirms that the additive polynomial $A(x)\in \mathbb K_\infty[x]$ has the properties claimed in the statement 
of Theorem \ref{theorem5.3}, and completes the proof.
\end{proof}

\bibliographystyle{amsbracket}
\providecommand{\bysame}{\leavevmode\hbox to3em{\hrulefill}\thinspace}

\end{document}